\documentclass[12pt]{article}

\usepackage[fleqn]{amsmath}
\usepackage{amsthm,amssymb}
\usepackage{graphicx}
\usepackage{tikz}
\usepackage{color,url}
\usepackage{float} 

\newtheorem{theorem}{Theorem}[section]
\newtheorem{lemma}[theorem]{Lemma}
\newtheorem{corollary}[theorem]{Corollary}
\newtheorem{proposition}[theorem]{Proposition}

\tikzstyle{noeud}=[circle,inner sep=2, minimum size =3 pt, line width = 1pt, draw=black, fill=white]
\tikzstyle{noeud2}=[circle,inner sep=1.4, minimum size =3 pt, line width = 1pt, draw=black, fill=white]

\DeclareMathOperator {\sg} {sg}

\DeclareMathOperator {\diam} {diam}

\begin{document}

\title{Strong geodetic number of complete bipartite graphs, crown graphs and hypercubes}

\author{
    Valentin Gledel $^{a}$
    \and
	Vesna Ir\v si\v c $^{b,c}$
}

\date{\today}

\maketitle
\begin{center}
	$^a$ Univ Lyon, Universit\'e Lyon 1, LIRIS UMR CNRS 5205, F-69621, Lyon, France \\
	\medskip

	$^b$ Institute of Mathematics, Physics and Mechanics, Ljubljana, Slovenia\\
	\medskip

	$^c$ Faculty of Mathematics and Physics, University of Ljubljana, Slovenia\\
		
\end{center}

\begin{abstract}
The strong geodetic number, $\text{sg}(G),$ of a graph $G$ is the smallest number of vertices such that by fixing one geodesic between each pair of selected vertices, all vertices of the graph are covered. In this paper, the study of the strong geodetic number of complete bipartite graphs is continued. The formula for $\text{sg}(K_{n,m})$ is given, as well as a formula for the crown graphs $S_n^0$. Bounds on $\text{sg}(Q_n)$ are also discussed.
\end{abstract}


\medskip\noindent
{\bf AMS Subj.\ Class:} 05C12, 05C70

\section{Introduction}
\label{sec:intro}

The three mainly studied variations of covering vertices of a graph with shortest paths (also called \emph{geodesics}) are the geodetic problem~\cite{bresar-2008, bresar-2011, bueno-2018, chartrand-2000, dourado-2010, hansen-2007, HLT93, jiang-2004,PanChang-2006,soloff-2015}, the isometric path problem~\cite{clarke-2008, Fisher2001, fitzpatrick-1999}, and the strong geodetic problem. The latter aims to determine the smallest number of vertices needed, such that by fixing one geodesic between each pair of selected vertices, all vertices of a graph are covered. More formally, the problem is introduced in~\cite{MaKl16a} as follows. 

Let $G=(V,E)$ be a graph. Given a set $S\subseteq V$, for each pair of vertices $\{x,y\}\subseteq S$, $x\ne y$, let $\widetilde{g}(x,y)$ be a {\em selected fixed} shortest path between $x$ and $y$. We set  
$$\widetilde{I}(S)=\{\widetilde{g}(x, y) : x, y\in S\}\,,$$ and $V(\widetilde{I}(S))=\bigcup_{\widetilde{P} \in \widetilde{I}(S)} V(\widetilde{P})$. If $V(\widetilde{I}(S)) = V$ for some $\widetilde{I}(S)$, then the set $S$ is called a {\em strong geodetic set}. For a graph $G$ with just one vertex, we consider the vertex as its unique strong geodetic set. The {\em strong geodetic problem} is to find a minimum strong geodetic set of $G$. The cardinality of a minimum strong geodetic set is the {\em strong geodetic number} of $G$ and is denoted by $\sg(G)$. Such a set is also called an \emph{optimal strong geodetic set}.

In the first paper on the topic~\cite{MaKl16a}, the strong geodetic number of complete Apollonian networks is determined and it is proved that the problem is NP-complete in general. Also, some comparisons are made with the isometric path problem. The problem has also been studied on grids and cylinders~\cite{Klavzar+2017}, and on Cartesian products in general~\cite{products}. Additional results about the problem on Cartesian products, as well as a notion of a strong geodetic core, has been recently studied in~\cite{cores}. Along with the strong geodetic problem, an edge version of the problem was also introduced in~\cite{MaKl16b}.

The strong geodetic problem appears to be difficult even on complete bipartite graphs. Some initial investigation is done in~\cite{bipartite}, where the problem is presented as an optimization problem and the solution is found for balanced complete bipartite graphs. Some more results have been very recently presented in~\cite{bipartite2}, where it is proved that the problem is NP-complete on general bipartite graphs, but polynomial on complete bipartite graphs. The asymptotic behavior of the strong geodetic problem on them is also discussed. 

In this paper we continue the study on bipartite graphs, specifically on the complete bipartite graphs, crown graphs, and hypercubes. In Section~\ref{sec:complete_bipartite}, we determine the explicit formula for complete bipartite graphs. In Section~\ref{sec:matching}, we discuss the strong geodetic number of crown graphs. In the last section, an upper and lower bound for the strong geodetic number of hypercubes are investigated. 

To conclude this section we state some basic definitions. Recall that a \emph{crown graph} $S_n^0$ is a complete bipartite graph $K_{n,n}$ without a perfect matching. Recall also that a \emph{hypercube} $Q_n$ is a graph on the vertex set $\{0, 1\}^n$, where two vertices are adjacent if and only if they differ in exactly one bit.

\section{Complete bipartite graphs}
\label{sec:complete_bipartite}

Strong geodetic number of complete bipartite graphs has been widely studied, as mentioned in Section~\ref{sec:intro}. For completeness we state some already known results.

\begin{theorem}[\cite{bipartite}, Theorem 2.1]
	\label{thm:balanced}
	If $n \geq 6$, then
	$$\sg(K_{n,n}) = \begin{cases}
	2 \left \lceil \displaystyle \frac{-1 + \sqrt{8 n + 1}}{2} \right \rceil, & 8n - 7 \text{ is not a perfect square},\vspace{0.2cm}\\
	2 \left \lceil \displaystyle \frac{-1 + \sqrt{8 n + 1}}{2} \right \rceil - 1, & 8n - 7 \text{ is a perfect square}.
	\end{cases}$$
\end{theorem}

Note also that $\sg(K_{1,m}) = m$ for all positive integers $m$, as $K_{1,m}$ is a tree with $m$ leaves. Hence, in the following results, this case is omitted.

To determine $\sg(K_{n,m})$, we will need the following definitions and notation. We will denote the extension of an integer valued function $\varphi$ to the real values by $\widetilde{\varphi}$. Let $3 \leq n \leq m$ be integers. Define $g(p) = m - \binom{p}{2}$ for $p \in \{0\ldots, n\}$ and $\widetilde{g}(p) = m - \binom{p}{2}$ for $p \in \mathbb{R}$ as a continuous extension of $g$. 

For $p \in \{0, 1, \ldots, n\}$ define $f(p) = \min \{ q \in \mathbb{Z} : \ \binom{q}{2} \geq n-p \}$ and its continuous extension $\widetilde{f} (p) = \frac{1 + \sqrt{1 + 8 (n-p)}}{2}$, the solution of $\binom{q}{2} = n-p$, where $p, q \in \mathbb{R}$. Observe that for $p < n$, $f(p) = \lceil \widetilde{f}(p) \rceil$, but $0 = f(n) \neq \lceil \widetilde{f} (n) \rceil = 1$.  

Define also $F(k) = k + f(k)$, $G(k) = k + g(k)$ and $s(k) = \max \{ F(k), G(k) \}$, for all $k \in \{0, 1, \ldots, n\}$. 
Note that whenever the functions defined above are used, the integers $n$ and $m$ will be clear from the context.

\begin{lemma}
	\label{lem:sg}
	If $3 \leq n \leq m$, then $\sg(K_{n,m}) = \min \{ s(k): \ 0 \leq k \leq n, k \in \mathbb{Z} \}$.
\end{lemma}

\begin{proof}
Let $(X, Y)$, $|X| = n$, $|Y| = m$, be the bipartition of $K_{n,m}$. Let $S_k$ be a minimal strong geodetic set of the graph, which has exactly $k$ vertices in $X$. Denote $l = |S_k \cap Y|$. As $Y$ must be covered, $l \geq m - \binom{k}{2} = g(k)$ (vertices of $Y$ are covered by being in a strong geodetic set or by a geodesic of length two between two vertices in the strong geodetic set in $X$). 

As $X$ must also be covered, $l$ must be such that $\binom{l}{2} \geq n - k$. Thus by definition of $f$, $l \geq f(k)$.

If $l \geq g(k)$ and $l \geq f(k)$, then both $X$ and $Y$ are covered. Hence by the minimality of $S_k$, we have $l = \max \{f(k), g(k) \}$.

Thus 
\begin{align*}
	\sg(K_{n,m}) & = \min \{ |S_k| : \ 0 \leq k \leq n \}  \\
	& = \min \{ k + \max \{f(k), g(k) \} : \ 0 \leq k \leq n \}  \\
	& = \min \{s(k) : \ 0 \leq k \leq n \} \,.  \qedhere
\end{align*}
\end{proof}

The main idea behind the following result is that $\sg(K_{n,m})$ is probably close to the value of $\min \{\max \{k + \widetilde{f}(k), k + \widetilde{g}(k)\} : \ 0 \leq k \leq n \}$. But before we state the more general result, consider the case $n = 2$, which has already been studied in~\cite[Corollary 2.3]{bipartite2}.

\begin{proposition}
	If $m \geq 2$, then $\sg(K_{2,m}) = 
	\begin{cases}
	3 ; & m = 2,\\
	m ; & m \geq 3.
	\end{cases}
	$
\end{proposition}


\begin{theorem}
	\label{thm:main}
	If $3 \leq n \leq m$, then 
	$$\sg(K_{n,m}) = \begin{cases}
	m ; \; & n-3 \geq \binom{m-3}{2}\, ,\\
	m + n - \binom{n}{2} ; \; & m \geq \binom{n}{2} \, ,\\
	n ; \; & \binom{n}{2} > m \geq 3 + \binom{n-3}{2} \, , \\
	\min \left\{ G(\lceil x^* \rceil - 1), F(\lceil x^* \rceil) \right \}; \; & \text{otherwise} \, ,
	\end{cases}$$
	where $3 \leq x^* \leq n-3$ is a solution of $m - \binom{x}{2} = \frac{1 + \sqrt{1 + 8(n-x)}}{2}$. 
\end{theorem} 

Observe that the first case simplifies to $(n,m) \in \{ (3,3),$ $(3,4),$ $(4,4),$ $(4,5),$ $(5,5),$ $(6,6) \}$  and that the ``otherwise'' case appears if and only if $m \geq n \geq 7$ and $m \leq 3 + \binom{n-3}{2}$. Note that the second case is indeed a known result from~\cite[Corollary 2.3]{bipartite2}, but here we present a different proof for it.

As the proof consists of some rather technical details, we shall first prove some useful lemmas. 

\begin{lemma}
	\label{lem:G}
	For $3 \leq k \leq n$, the function $G(k)$ is strictly decreasing. Moreover, we have $G(0) = G(3) = m$ and $G(1) = G(2) = m+1$. The same holds for the function $\widetilde{g}(k)$.
\end{lemma}

\proof 
It follows from the fact that $k - \binom{k}{2}$ is strictly decreasing for $k \geq 3$.
\qed

\begin{lemma}
	\label{lem:F}
	For $0 \leq k \leq n-1$, $F(k)$ and $\widetilde{f} (k)$ are increasing (not necessarily strictly). Additionally, we have $F(n-1) = n+1$ and $F(n) = n$. Moreover, $|F(k+1) - F(k)| \leq 1$ for all $0 \leq k \leq n-1$. 
\end{lemma}

\proof
The claim follows from the fact that if $0 \leq p \leq n-1$ and $f(p) = q$, then $f(p+1) \in \{q, q-1\}$ by definition of $f$. 
\qed

\begin{lemma}
	\label{lem:intersection}
	If $n-3 < \binom{m-3}{2}$, $m -3 \leq \binom{n-3}{2}$ and $n \geq 3$, then the functions $\widetilde{f}$ and $\widetilde{g}$ intersect exactly once on $[3,n-3]$. 
	The condition is equivalent to $(n,m) \in  \{ (n,m) : \ n \geq 7, n \leq m \leq 3 + \binom{n-3}{2} \}$.
\end{lemma}

\proof
The simplification of the conditions can be checked by a simple calculation.

Moreover, the first condition implies $\widetilde{f}(3) \leq \widetilde{g}(3)$, and the second implies $\widetilde{f}(n-3) \geq \widetilde{g}(n-3)$. As $\widetilde{f}$ is increasing and $\widetilde{g}$ is strictly decreasing on $[3,n-3]$ (and $n \geq 7$), it follows that $\widetilde{f}$ and $\widetilde{g}$ intersect exactly once on $[3,n-3]$.
\qed

\begin{lemma}
	\label{lem:properties}
	For $m \geq n \geq 7$ and $m \leq 3 + \binom{n-3}{2}$, let $\widetilde{f}$ and $\widetilde{g}$ intersect in $x^* \in [3,n-3]$. Then it holds:\\
	(i) $F(\lceil x^* \rceil) \geq G(\lceil x^* \rceil)$,\\
	(ii) $G(\lceil x^* \rceil-1) \geq F(\lceil x^* \rceil-1)$.
\end{lemma}

\proof
As we are on the interval $[3,n-3]$ it suffices to prove both properties for $\widetilde{f}$ and $\widetilde{g}$ instead of $F$ and $G$, as $f(\lceil x \rceil) = \widetilde{f}(\lceil x \rceil)$ for $x \in [3, n-3]$ and similarly for $g$ and $\widetilde{g}$. 

(i) As $\widetilde{f}$ is increasing, $\widetilde{g}$ is strictly decreasing, $\widetilde{f}(x^*)=\widetilde{g}(x^*)$, and $x^* \leq \lceil x^* \rceil$, it follows that $\widetilde{f}(\lceil x^* \rceil) \geq \widetilde{g}(\lceil x^* \rceil)$.

(ii) As $\widetilde{f}$ is increasing, $\widetilde{g}$ is strictly decreasing, $\widetilde{f}(x^*)=\widetilde{g}(x^*)$, and $\lceil x^* \rceil - 1 \leq x^*$, it follows that $\widetilde{f}(\lceil x^* \rceil - 1) \leq \widetilde{f}(x^*) = \widetilde{g}(x^*) < \widetilde{g}(\lceil x^* \rceil - 1)$.
\qed

\begin{proof}[Proof of Theorem~\ref{thm:main}] Recall that by Lemma~\ref{lem:sg}, $\sg(K_{n,m}) = \min \{ s(k): \ 0 \leq k \leq n, k \in \mathbb{Z} \}$ $=$ $\min \{ \max\{F(k), G(k)\}: \ 0 \leq k \leq n, k \in \mathbb{Z} \}$.

\textit{Case 1:} Let $n-3 \geq \binom{m-3}{2}$. The only possibilities are $(n,m) \in \{ (3,3),$ $(3,4),$ $(4,4),$ $(4,5),$ $(5,5),$ $(6,6) \}$. For all of them we can easily check that the optimal value is $m$.  

\textit{Case 2:} Let $m \geq \binom{n}{2}$. Thus, $G(n) \geq F(n)$. As $n \geq 3$ and $G$ is strictly decreasing, it also holds $G(n-1) \geq F(n-1)$. If $n = 3$, then $f(0) = f(3) = 3$ and $f(4) = f(5) = 4$. Thus the minimum is attained in either $k = 0$ or $k = n$, but the value is the same in both cases and equals $m + n - \binom{n}{2} = m$. If $n \geq 4$, then $G(k) \geq F(k)$ for all $0 \leq k \leq n$ (by the properties of $F$ and $G$) and the minimum is attained in $k = n$, hence the value is $m + n - \binom{n}{2}$.

\textit{Case 3:} Let $\binom{n}{2} > m \geq 3 + \binom{n-3}{2}$. Thus $G(n) < F(n)$ and $G(n-3) \geq F(n-3)$. Hence by the properties of $G$ and $F$, the minimum is attained in $k = n$, hence the value is $F(n) = n$. 

\textit{Case 4:} Lastly, we study the ``otherwise'' case, i.e.\ the case when $m \geq n \geq 7$ and $m \leq 3 + \binom{n-3}{2}$. By Lemma~\ref{lem:intersection}, there exists an $x^* \in [3, n-3]$ such that $\widetilde{f}(x^*) = \widetilde{g} (x^*)$. Clearly, $\min \max \{\widetilde{f}, \widetilde{g}\}$ on $[0, n]$ is attained in $x^*$. But as $x^*$ is not necessarily an integer, we must further investigate the properties of $F$ and $G$ on integer values close to $x^*$. By Lemma~\ref{lem:properties}(i), $F(\lceil x^* \rceil) \geq G(\lceil x^* \rceil)$, thus also $F(k) \geq G(k)$ for all $k \geq \lceil x^* \rceil$. By Lemma~\ref{lem:properties}(ii), $G(\lceil x^* \rceil-1) \geq F(\lceil x^* \rceil-1)$, thus $G(k) \geq F(k)$ for all $k \leq \lceil x^* \rceil-1$. Hence, $$\min \{ \max \{ F(k), G(k) \}: \ 0 \leq k \leq n, k \in \mathbb{Z} \} = \min \left\{ G(\lceil x^* \rceil - 1), F(\lceil x^* \rceil) \right \}\, . \qedhere$$ 
\end{proof}

As already mentioned, asymptotic behavior of $\sg(K_{n,m})$ is presented in~\cite[Theorem 2.5]{bipartite2}. Two special cases of this behavior are a direct consequence of the second and third case in the above Theorem~\ref{thm:main}, but determining the asymptotic behavior for a general case is not trivial even with the result of Theorem~\ref{thm:main}.

\section{Crown graphs}
\label{sec:matching}

In the following we determine $\sg(S_n^0)$, using similar techniques as in the proof of Theorem~\ref{thm:balanced}~\cite{bipartite}.

\begin{lemma}
	\label{lem:max1}
	Let $T = T_1 \cup T_2$ be a strong geodetic set of $S_n^0$, $n \geq 2$, with bipartition $(X, Y)$, where $T_1 \subseteq X$, $T_2 \subseteq Y$ and $t_i = |T_i|$ for all $i \in \{1,2\}$. If $|t_1 - t_2| \geq 2$, then there exists a strong geodetic set $T' = T'_1 \cup T'_2$, $T'_1 \subseteq X$, $T'_2 \subseteq Y$, such that $|T'| = |T|$ and $|t'_1 - t'_2| < |t_1 - t_2|$, where $t'_i = |T'_i|$ for $i \in \{1,2\}$.
\end{lemma}

\proof
Let $X = \{ x_1, \ldots, x_n \}$, $Y = \{ y_1, \ldots, y_n \}$ and $x_i \sim y_i$ for $i \in [n]$ be a removed perfect matching. 
Without loss of generality, we assume $t_1 \geq t_2$ and let $t_1 - t_2 = k \geq 2$, $T_1 = \{x_1, \ldots, x_{t_2 + k}\}$ and $T_2 = \{ y_1, \ldots, y_{t_2} \}$. Note that geodesics between $x_{t_2+k}$ and $T_2$ are just edges. 

First consider the case $t_2 = 0$. As $T$ is a strong geodetic set, $t_1 = n$. Consider $T' = (X - \{x_n\}) \cup \{y_1\}$. Clearly, $|T'| = |T|$ and $n-2 = |t_1' - t_2'| < |t_1 - t_2| = n$. To see that $T'$ is a strong geodetic set, fix the following geodesics: $x_1 \sim y_2 \sim x_n \sim y_1$ and $x_1 \sim y_{i+1} \sim x_i$ for all $2 \leq i \leq n-1$. 

Next consider the case $t_2 = \min\{t_1, t_2\} \geq 1$. Define $T' = T_1' \cup T_2'$ where $T_1' = T_1 - \{x_{t_2 + k}\}$ and $T_2' = T_2 \cup \{y_{t_2 + 1}\}$. As $t_2 \geq 1$, we have $0 \leq t_i' \leq n$. Clearly, $|T'| = |T|$ and $|t_1' - t_2'| < |t_1 - t_2|$. In the following we prove that $T'$ is a strong geodetic set. First cover $x_{t_2 + k - 1}$ by a geodesic $y_1 \sim x_{t_2 + k} \sim y_{t_2+1}$ (as $t_2 \geq 1$ and $k \geq 2$). Next fix the following geodesics: $x_i \sim y_{t_2 + k + i} \sim x_{t_2 + k + i+1} \sim y_i$ for $i \in [t_2 - 1]$ and $x_{t_2} \sim y_{t_2 + k + t_2} \sim x_{t_2 + k + 1} \sim y_{t_2}$, and $\binom{t_2}{2}$ geodesics between vertices in $\{x_1, \ldots, x_{t_2}\}$ and also between $\{y_1, \ldots, y_{t_2}\}$ to cover vertices $\{ x_{t_2+k+t_2+1}, \ldots, x_{t_2 + k + t_2 + \binom{t_2}{2}} \} \cup \{ y_{t_2+k+t_2+1}, \ldots, y_{t_2 + k + t_2 + \binom{t_2}{2}} \}$. If $n < t_2 + k + t_2 + \binom{t_2}{2}$, fix geodesics in a similar manner (but in this case not all are needed). On the other hand, $t_2 + k + \binom{t_2}{2} + t_2 \geq n$ as $T_1 \cup T_2$ is a strong geodetic set and thus covers $X$. The only uncovered vertices are then $\{y_{t_2 + 2}, \ldots, y_{t_2+k}\}$, hence $k-1$ vertices in $Y$. They can be covered by the (not yet used) geodesics $x_1 \sim y_{i+1} \sim x_i$ for $i \in \{t_2+1, \ldots, t_2+k-1\}$. 
\qed

\begin{theorem}
	\label{thm:matching}
	If $n \geq 2$, then
	$$\sg(S_n^0) = \begin{cases}
	2 \left \lceil \displaystyle \frac{-3 + \sqrt{8 (n+1) + 1}}{2} \right \rceil, & 8n + 1 \text{ is not a perfect square},\vspace{0.2cm}\\
	2 \left \lceil \displaystyle \frac{-3 + \sqrt{8 n + 1}}{2} \right \rceil + 1, & 8n + 1 \text{ is a perfect square}.
	\end{cases}$$
	For an optimal strong geodetic set $S = S_1 \cup S_2$ it holds $|S_1| = |S_2| = \left \lceil \frac{-3 + \sqrt{8 n + 9}}{2} \right \rceil$, if $8n + 1$ is not a perfect square, otherwise $|S_1| = \left \lceil \frac{-3 + \sqrt{8 n + 1}}{2} \right \rceil$ and $|S_2| = |S_1| + 1$.
\end{theorem}

\proof
Let $S = S_1 \cup S_2$ be a strong geodetic set of $S_n^0$ where $|S| = \sg(S_n^0)$ and $S_1$, $S_2$ each lie in one part of the bipartition. Let $p = |S_1|$ and $q = |S_2|$. Without loss of generality, $p \leq q$.

As $S$ is a strong geodetic set and the geodesics in $S_n^0$ are either edges (between a vertex in $S_1$ and a vertex in $S_2$), paths of length 2 (between two vertices in $S_i$) or paths of length 3 (between unconnected vertices in $S_1$ and $S_2$), it holds:
$$n \leq p + \binom{q}{2} + p \, ,$$
$$n \leq q + \binom{p}{2} + p \, .$$
From Lemma~\ref{lem:max1} it follows that we can assume $|p-q| \leq 1$. Hence, we distinguish two cases.

\textit{Case 1:} Let $p = q$. The inequalities simplify to $n \leq 2p + \binom{p}{2}$ for $0 \leq p \leq n$, which is equivalent to $p \geq  \frac{-3 + \sqrt{8 n + 9}}{2}$. Hence, the optimal value is $p = q = \left \lceil \frac{-3 + \sqrt{8 n + 9}}{2} \right \rceil$ and $|S| = 2 p$. 

\textit{Case 2:} Let $q = p + 1$. The inequalities simplify to $n \leq 2 p + \binom{p+1}{2} = 3p + \binom{p}{2}$ and $n \leq 2 p + 1 + \binom{p}{2}$. As $n \geq 2$, we have $p \geq 1$, so the second inequality gives a stronger constraint. Hence, $p \geq  \frac{-3 + \sqrt{8 n + 1}}{2}$. The optimal value is $p = \left \lceil \frac{-3 + \sqrt{8 n + 1}}{2} \right \rceil$, $q = p+1$ and $|S| = 2p+1$. 

Next, we determine the strong geodetic number of $S_n^0$, i.e.\ determine which case gives rise to a smaller value of $p+q$. Let $a = \left \lceil \frac{-3 + \sqrt{8 n + 9}}{2} \right \rceil$, $b = \left \lceil \frac{-3 + \sqrt{8 n + 1}}{2} \right \rceil$, $\sg_a = 2 a$ and $\sg_b = 2 b + 1$. We distinguish two cases: $8 n + 1$ is a perfect square or not.

\textit{Case 1:} If $8n+1$ is a perfect square. There exists an integer $m$ such that $m^2 = 8n+1$. Thus $m$ is odd and $m \geq 5$ (as $n \geq 2$). Hence, $\left \lceil \frac{-3 + m}{2} \right \rceil = \frac{m-3}{2}$ and $\sg_b = m-2$. On the other hand, $m^2 < 8n+9 \leq (m+1)^2$, thus $a \leq \left \lceil \frac{-3 + m+1}{2} \right \rceil = \frac{m-1}{2}$ and $\sg_a = m-1$. In this case $\sg_b < \sg_a$.

\textit{Case 2:} If $8n + 1$ is not a perfect square. There exists an integer $m$ such that $m^2 < 8n+1 < (m+1)^2$. Notice that $m \geq 4$ as $n \geq 2$. Thus $$b = \left \lceil \frac{-3 + m+1}{2} \right \rceil = \begin{cases}
\frac{m-1}{2}, & m \text{ odd},\\
\frac{m-2}{2}, & m \text{ even},
\end{cases}$$
and $$\sg_b = \begin{cases}
m, & m \text{ odd},\\
m-1, & m \text{ even}.
\end{cases}$$ 
On the other hand, $8n+9 < (m+1)^2 + 8 < (m+2)^2$ as $m \geq 4$. Thus $$a \leq \left \lceil \frac{-3 + m+2}{2} \right \rceil = \begin{cases}
\frac{m-1}{2}, & m \text{ odd},\\
\frac{m}{2}, & m \text{ even},
\end{cases}$$
and $$\sg_a \leq \begin{cases}
m-1, & m \text{ odd},\\
m, & m \text{ even}.
\end{cases}$$ 
Hence, if $m$ is odd, $\sg_a < \sg_b$. But if $m$ is even, then $m+1$ is odd and by~\cite[Lemma 2.2]{bipartite}, there exists an integer $k$ such that $(m+1)^2 = 8k+1$. Clearly, $8n+9 \leq (m+1)^2 + 7$, but due to the congruences modulo $8$, we conclude that $8n+9 \leq (m+1)^2$. Thus $a \leq \left \lceil \frac{-3 + m+1}{2} \right \rceil = \frac{m-2}{2}$ and $\sg_a = m-2$. Hence in this case we also have $\sg_a < \sg_b$, which concludes the proof. 
\qed

A very special case of a complete bipartite graph without a perfect matching is a cube $Q_3 \cong S_4^0$. By Theorem~\ref{thm:matching}, $\sg(Q_3) = 4$. In the next section, we study hypercubes more thoroughly.

\section{Hypercubes}
\label{sec:cubes}

In the last section, we discuss the strong geodetic problem on another family of bipartite graphs, namely hypercubes. It is known that the strong geodetic problem on bipartite graphs is NP-complete~\cite{bipartite2}, thus it would be optimistic to expect an explicit formula. Recall that a hypercube $Q_n$ has $2^n$ vertices and diameter $n$. First, we consider some small hypercubes.

Clearly, $\sg(Q_0) = 1$, $\sg(Q_1) = 2$ and $\sg(Q_2) = 3$. We already know that $\sg(Q_3) = 4$. Using a computer program, we can also check that $\sg(Q_4) = 5$. For an example of the smallest strong geodetic sets, see Figure~\ref{fig:small_cubes}.

\begin{figure}[ht]
	\begin{center}
		\begin{tikzpicture}[scale=0.8]
		
		\begin{scope}
		\node[noeud, fill=black] (000) at (0,0){};
		\node[noeud] (001) at (3,0){};
		\node[noeud] (010) at (0,3){};
		\node[noeud] (011) at (3,3){};
		\node[noeud] (100) at (1,1){};
		\node[noeud, fill=black] (101) at (2,1){};
		\node[noeud, fill=black] (110) at (1,2){};
		\node[noeud, fill=black] (111) at (2,2){};
		
		\draw (000) -- (001) -- (011) -- (010) -- (000) -- (100) -- (101) -- (111) -- (110) -- (100);
		\draw (010) -- (110);
		\draw (011) -- (111);
		\draw (001) -- (101);
		\end{scope}
		
		\begin{scope}[xshift=5cm]
		\node[noeud] (0000) at (0,0){};
		\node[noeud, fill=black] (0001) at (3,0){};
		\node[noeud] (0010) at (0,3){};
		\node[noeud] (0011) at (3,3){};
		\node[noeud] (0100) at (1,1){};
		\node[noeud] (0101) at (2,1){};
		\node[noeud] (0110) at (1,2){};
		\node[noeud] (0111) at (2,2){};
		
		\node[noeud, fill=black] (1001) at (5,0){};
		\node[noeud] (1000) at (8,0){};
		\node[noeud] (1011) at (5,3){};
		\node[noeud] (1010) at (8,3){};
		\node[noeud] (1101) at (6,1){};
		\node[noeud, fill=black] (1100) at (7,1){};
		\node[noeud, fill=black] (1111) at (6,2){};
		\node[noeud, fill=black] (1110) at (7,2){};
		
		\draw (0000) -- (0001) -- (0011) -- (0010) -- (0000) -- (0100) -- (0101) -- (0111) -- (0110) -- (0100);
		\draw (0010) -- (0110);
		\draw (0011) -- (0111);
		\draw (0001) -- (0101);
		
		\draw (1000) -- (1001) -- (1011) -- (1010) -- (1000) -- (1100) -- (1101) -- (1111) -- (1110) -- (1100);
		\draw (1010) -- (1110);
		\draw (1011) -- (1111);
		\draw (1001) -- (1101);
		
		\draw (0001) -- (1001);
		\draw (0011) -- (1011);
		\draw (0101) -- (1101);
		\draw (0111) -- (1111);
		
		\draw (0000) to [out=-20,in=-160] (1000);
		\draw (0010) to [out=20,in=160] (1010);
		\draw (0100) to [out=-20,in=-160] (1100);
		\draw (0110) to [out=20,in=160] (1110);
		\end{scope}

		\end{tikzpicture}
	\end{center}
	\caption{Hypercubes $Q_3$ and $Q_4$ with their strong geodetic sets.}
	\label{fig:small_cubes}
\end{figure}
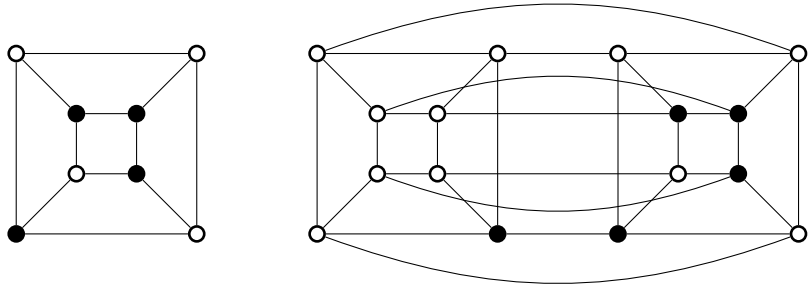

We can use the result from~\cite[Theorem 3.1]{bipartite} to attain a lower bound for $\sg(Q_n)$. The result states that if $G$ is a graph with $n = n(G)$ and $d = \diam(G) \geq 2$, then 
$$\sg(G) \geq \left \lceil \frac{d - 3 + \sqrt{(d-3)^2 + 8 n (d-1)}}{2 (d-1)} \right \rceil.$$
Using this we obtain

\begin{corollary}
	\label{cor:cubes_lower}
	If $n \geq 2$, then $$\sg(Q_n) \geq \left \lceil \frac{1}{\sqrt{n-1}} \cdot 2^{\frac{n+1}{2}} \right \rceil.$$
\end{corollary}

\begin{proof}
	Using~\cite[Theorem 3.1]{bipartite}, we get $$\sg(Q_n) \geq \left \lceil \frac{n - 3 + \sqrt{(n-3)^2 + 2^{n+3} (n-1)}}{2 (n-1)} \right \rceil \geq \left \lceil \frac{1}{\sqrt{n-1}} \cdot 2^{\frac{n+1}{2}} \right \rceil,$$
	which concludes the proof.
\end{proof}

On the other hand, we present a non-trivial upper bound, stating that approximately a square root of the number of vertices is enough to form a strong geodetic set.

\begin{theorem}
	\label{thm:cubes_upper}
	If $n \geq 1$, then 
	$$\sg(Q_n) \leq \begin{cases}
	 \frac{3}{2} \cdot 2^{\frac{n}{2}}, & n \text{ is even},\\
	2^{\frac{n+1}{2}}, & n \text{ is odd}.
	\end{cases}$$
\end{theorem}

\begin{proof}
First we prove the intermediate result that for all $n_0$, $n \geq n_0 \geq 1$, it holds that $\sg(Q_n) \leq 2^{n - n_0} + 2^{n_0 - 1}$. 
	
Let $n_0$ be an integer, $n \geq n_0 \geq 1$. Denote $0^{k}$ as a string of $k$ zeros and $1^{k}$ as a string of $k$ ones. A hypercube $Q_n$ consists of $2 \cdot 2^{n-n_0}$ copies of hypercubes $Q_{n_0 - 1}$. These copies are labeled as $Q_{n_0 - 1}^b$, where $b \in \{0, 1\}^{n - n_0 + 1}$ and the vertices of the graph $Q_n$ are of the form $bc$, $b \in \{0, 1\}^{n - n_0 + 1}$, $c \in \{0, 1\}^{n_0 - 1}$.

Let $P =  \{ \overline{b}0 0^{n_0-1} :  \overline{b} \in \{0, 1\}^{n - n_0}\}$, $Q = \{ 1^{n-n_0}1c : \ c \in \{0, 1\}^{n_0 - 1} \} = V(Q_{n_0-1}^{1^{n-n_0+1}})$, and $S = P \cup Q$. Notice that $|S| = 2^{n - n_0} + 2^{n_0 - 1}$. Next, we prove that $S$ is a strong geodetic set of $Q_n$.

For each pair of vertices $\overline{b}00^{n_0 - 1} \in P$ and $1^{n-n_0}1c \in Q$ we fix the following geodesic (where $\rightsquigarrow$ denotes some shortest path between given vertices): $$\overline{b}00^{n_0 - 1} \rightsquigarrow \overline{b}0c \sim \overline{b}1c \rightsquigarrow 1^{n-n_0}1c\,.$$ As $\overline{b}$ and $c$ can be any strings of zeros and ones of the appropriate length, all vertices of the hypercube $Q_n$ are covered. Hence for all $n_0$, $n \geq n_0 \geq 1$, $$\sg(Q_n) \leq 2^{n - n_0} + 2^{n_0 - 1}\,.$$

Therefore, $\sg(Q_n) \leq \min\{2^{n - n_0} + 2^{n_0 - 1} : \ n_0 \in \mathbb{N}, n \geq n_0 \geq 1\}$. The minimum of this function for $n_0 \in \mathbb{R}$ is attained in $n_0 = \frac{n+1}{2}$. Thus the minimum of the integer valued function is in $\frac{n+1}{2}$ if $n$ is odd, and in either $\lfloor \frac{n+1}{2} \rfloor$ or $\lceil \frac{n+1}{2} \rceil$ if $n$ is even. If $n$ is odd, then the minimal value is $2^{\frac{n+1}{2}} = \sqrt{2} \cdot 2^{\frac{n}{2}}$. If $n$ is even, the value in both cases is $\frac{3}{2} \cdot 2^{\frac{n}{2}}$ and thus this is the minimal value. 
\end{proof}

By being more careful when selecting the geodesics, we can improve the bound as follows.

\begin{theorem}
	\label{thm:cubes_upper2}
	If $n \geq 2$, then 
	$$\sg(Q_n) \leq \begin{cases}
	\frac{3}{2} \cdot 2^{\frac{n}{2}} - \left( \lceil \frac{n+1}{2} \rceil - 2 \right) \left( \lceil \frac{n+1}{2} \rceil - 3 \right) , & n \text{ is even},\\
	2^{\frac{n+1}{2}} - \frac{(n-3)(n-5)}{4}, & n \text{ is odd}.
	\end{cases}$$
\end{theorem}

\begin{proof}
	Using the notation from the proof of Theorem~\ref{thm:cubes_upper}, the main step is to prove that for all $n_0$, $n \geq n_0 \geq 4$, $$\sg(Q_n) \leq 2^{n - n_0} + 2^{n_0 - 1} - (n_0 - 2) (n_0-3) \,.$$
	From this it follows that by using $n_0 = \lceil \frac{n+1}{2} \rceil$, the result is obtained and the bound from Theorem~\ref{thm:cubes_upper} is improved if $\lceil \frac{n+1}{2} \rceil \geq 4$, ie.\ $n \geq 6$.
	
	Let $v, u \in V(Q_{n_0-1}^{1^{n-n_0+1}})$ be some vertices at distance $n_0$. Without loss of generality, we take $v = 1^{n-n_0+1}0^{n_0-1}$ and $u = 1^{n}$. There is $n_0-1$ internally disjoint shortest paths from $v$ to $u$, set $\mathcal{P}$ as the set of vertices they cover. Let $x_1, \ldots, x_{n_0-1}$ be the neighbors of $u$ on these paths and $y_1, \ldots, y_{n_0-1}$ the other neighbors of $x_i$'s on these paths. Let $F = \mathcal{P} - \{u, v, x_1, \ldots, x_{n_0-1}, y_2, \ldots, y_{n_0-1}\} $.
	
	Now let $P =  \{ \overline{b}0 0^{n_0-1} :  \overline{b} \in \{0, 1\}^{n - n_0}\}$, $Q = \{ 1^{n-n_0}1c : \ c \in \{0, 1\}^{n_0 - 1} - F \} = V(Q_{n_0-1}^{1^{n-n_0+1}}) - F$, and $S = P \cup Q$. Clearly, $|S| = 2^{n - n_0} + 2^{n_0 - 1} - (n_0 - 2) - (n_0-4)(n_0 - 2) = 2^{n - n_0} + 2^{n_0 - 1} - (n_0 - 2) (n_0-3)$. Next, we prove that $S$ is a strong geodetic set of $Q_n$.
	
	For each pair of vertices $\overline{b}00^{n_0 - 1} \in P$ and $1^{n-n_0}1c \in Q$ we fix the following geodesic (where $\rightsquigarrow$ denotes some shortest path between given vertices, and this path follows $\mathcal{P}$ if the endpoints are appropriate): 
	$$\overline{b}00^{n_0 - 1} \rightsquigarrow \overline{b}0c \sim \overline{b}1c \rightsquigarrow 1^{n-n_0}1c, \quad c \in Q - \{x_1, y_2, \ldots, y_{n_0-1}\},$$
	$$\overline{b}00^{n_0 - 1} \rightsquigarrow  1^{n-n_0}10^{n_0 - 1} \rightsquigarrow  1^{n-n_0}1c, \quad c \in \{x_1, y_2, \ldots, y_{n_0-1}\}.$$
	As shortest paths in $\mathcal{P}$ get covered with the above geodesics, and other vertices are clearly covered, $S$ is a strong geodetic set.
\end{proof}

Some values given by the Theorem are presented in Table~\ref{tbl:values}. Note that asymptotically, the ratio between the lower and upper bound is of order $\frac{1}{\sqrt{n}}$.

\begin{table}[H]
	\centering
	\begin{tabular}{| c || c | c | c | c | c | c | c | c | c | c | c | c | c | c | c |}
		\hline
		$n$ & 1 & 2 & 3 & 4 & 5 & 6 & 7 & 8 & 9 & 10 & 11 & 12 & 13 & 14 & 15 \\ \hline \hline
		$\sg(Q_n) \geq $ &   & 3 & 3 & 4 & 4 & 6 & 7 & 9 & 12 & 16 & 21 & 28 & 37 & 51 & 69 \\ \hline \hline
		$\sg(Q_n) \leq $ &  & &  &  &  & 10 & 14 & 18 & 26 & 36 & 52 & 76 & 108 & 162 & 226 \\ \hline
	    $\sg(Q_n) \leq $ & 2 & 3 & 4 & 6 & 8 & 12 & 16 & 24 & 32 & 48 & 64 & 96 & 128 & 192 & 256 \\

		\hline
	\end{tabular}
	\caption{The lower and both upper bounds on $\sg(Q_n)$ given by Corollary~\ref{cor:cubes_lower}, Theorem~\ref{thm:cubes_upper2} and Theorem~\ref{thm:cubes_upper}.}
	\label{tbl:values}
\end{table}

It would be interesting to have an explicit formula for $\sg(Q_n)$ or at least know the complexity of determining the strong geodetic number of $Q_n$.

\section*{Acknowledgments}
\label{sec:Acknowledgments}

The authors would like to thank Sandi Klav\v zar and Matja\v z Konvalinka for a number of fruitful conversations.


\end{document}